\documentclass{amsart}

\usepackage{enumerate}
\usepackage{amssymb}
\usepackage{amsmath,amscd}
\usepackage{amsthm}
\usepackage{graphicx}
 \usepackage{verbatim}
 \usepackage{color}

\newcommand{\QQ}{\mathbb{Q}}

\newcommand{\PP}{\mathbb{P}}
\newcommand{\sph}{\mathbb{S}}
\newcommand{\CP}{\mathbb{CP}}
\newcommand{\HP}{\mathbb{HP}}

\newcommand{\ZZ}{\mathbb{Z}}

\newcommand{\SU}{\operatorname{SU}}
\newcommand{\Sp}{\operatorname{Sp}}
\newcommand{\Span}{\operatorname{span}}

\newtheorem{theorem}{Theorem}

\newtheorem{proposition}[theorem]{Proposition}

\newtheorem*{PWconj}{Petersen-Wilhelm's Conjecture}

\newtheorem{maintheorem}{Theorem}

\newtheorem{maincorollary}[maintheorem]{Corollary}

\theoremstyle{definition}

\theoremstyle{remark}

\newtheorem*{ack}{Acknowledgements}

\title[A note on the Petersen-Wilhelm conjecture]{A note on the Petersen-Wilhelm Conjecture}
\author[D.~Gonz\'alez-\'Alvaro]{David Gonz\'alez-\'Alvaro}
\address{Universidad Aut\'onoma de Madrid, Spain}
\email{dav.gonzalez@uam.es}

\author[M.~Radeschi]{Marco Radeschi}
\address{University of Notre Dame, US}
\email{mradesch@nd.edu}

\thanks{D. G.-\'A. received support from research grants MTM2014-57769-3-P and MTM2014-57309-REDT (MINECO)}%$^*$

\subjclass[2010]{53C21, 57R19}

\renewcommand{\arraystretch}{1.}

\begin{document}

\maketitle

\begin{abstract} In this note we consider submersions from compact manifolds, homotopy equivalent to the Eschenburg or Bazaikin spaces of positive curvature. We show that if the submersion is nontrivial, the dimension of the base is greater than the dimension of the fiber. Together with previous results, this proves the Petersen-Wilhelm Conjecture for all the known compact manifolds with positive curvature.
\end{abstract}
 
\section{Introduction}

At the present time there is a very small number of known methods to construct examples of manifolds with nonnegative or positive (sectional) curvature. Moreover, Riemannian submersions are present in the construction of almost all existing examples of positively curved manifolds (see \cite{Zil} for a survey on the topic). This is because, by the Gray-O'Neill formula, a lower bound for the sectional curvatures of the total space also bounds the curvatures of the base space. In the search of new examples of manifolds with positive curvature, understanding the behavior of Riemannian submersions under curvature-related assumptions seems crucial. In that direction, the following conjecture has been of great interest in the last decade.

\begin{PWconj}
If $M\to B$ is a Riemannian submersion between compact, positively curved manifolds, with fiber $F$, then $\dim F < \dim B$.
\end{PWconj}

It is worth mentioning that, although this conjecture has been attributed in the literature to Fred Wilhelm, during the preparation of this article the authors learned that it was originally considered jointly with Peter Petersen.

The conjecture is known to hold under the extra assumption that at least two of the fibers are totally geodesic, thanks to Frankel's theorem. On the other hand, there exist counterexamples if one weakens the conjecture to require positive sectional curvature on an open and dense set of the total space (cf. \cite{Ker} and \cite[p. 167]{Che}). We refer the reader to \cite{GAG,Sp} for recent progress towards the conjecture in the general case, and to \cite{FZ} for a complete list of all the (few) currently known Riemannian submersions between positively curved manifolds.

%The motivation behind this note is the potential construction of counterexamples to the Petersen-Wilhelm conjecture. As for the few known Riemannian submersions between positively curved spaces we refer the reader to \cite{FZ}. Moreover, a natural question is whether it is possible to construct a counterexample using as a total space one of the existing examples of manifolds with positive curvature. We shall show that this can not be the case, by proving that they do not even admit a smooth submersion violating Petersen-Wilhelm's bound.

The goal of this note is to show that every Riemannian submersion from any of the known compact positively curved manifolds satisfies the Petersen-Wilhelm conjecture. In fact, the results here will show that any \emph{submersion} (not necessarily Riemannian) from any compact manifold \emph{homotopy equivalent} to one of the known positively curved examples, satisfies the conjecture.

Recall that, by Bonnet-Myers Theorem, positively curved manifolds have finite fundamental group and therefore for our purposes it is enough to study simply connected examples. For the convenience of the reader, we briefly review all existing examples of simply connected manifolds with positive curvature.

\begin{itemize}
\item Spaces admitting a homogeneous metric of positive curvature. These have been classified (see \cite{WZ}) and include the classical compact rank one symmetric spaces, three single examples due to Wallach, two single examples due to Berger and an infinite family due to Aloff-Wallach: 
\[
\mathbb{S}^n, \CP^n, \HP^n, Ca\PP^2,\quad W^6,W^{12},W^{24},\quad B^7,B^{13},\quad W_{p,q}^7.
\]
\item Eschenburg \cite{Esc} constructed positively curved metrics on a biquotient space $E^6$ of dimension $6$; and on an infinite family of biquotients of dimension $7$, which are circle quotients $\SU(3)/\!/S^1$ and include the Aloff-Wallach spaces $W_{p,q}^7$ as a subfamily.

\item Bazaikin \cite{Baz} constructed positively curved metrics on an infinite family of biquotients $\SU(5)/\!/\Sp(2)S^1$ of dimension $13$, which contains the Berger space $B^{13}$ as a particular case.

\item Dearricot \cite{Dea} and Grove-Verdiani-Ziller \cite{GVZ} recently constructed a cohomogeneity one  positively curved manifold of dimension $7$, denoted $P_2$.
\end{itemize}

In a recent work, Amann and Kennard studied a generalized version of the Petersen-Wilhelm conjecture. Given a fibration $F\to M\to B$, they gave topological conditions for $M$ to ensure that $\dim F<\dim B$ (see Theorem A in \cite{AK}). These conditions are satisfied by any manifold that is rationally homotopy equivalent to any of the simply connected compact rank one symmetric spaces or to any of $W^6,W^{12},W^{24}$. Recall that $E^6$ is rationally homotopy equivalent to $W^6$, and that $B^7$ and $P_2$ are rational spheres. It follows that all of these examples satisfy (a stronger version of) the Petersen-Wilhelm conjecture. In this note we deal with the remaining examples, namely the Eschenburg biquotients of dimension $7$, and the Bazaikin biquotients, for which we prove the following results. 

%In particular, these conditions are satisfied by all compact rank one symmetric spaces, the spaces $W^6,W^{12},W^{24},E^6$ and, although they do not write it explicitly, by $B^7$ and $P_2$, as they are rational spheres. It follows that these examples satisfy (a stronger version of) the Petersen-Wilhelm conjecture. In this note, we prove that the same holds for the remaining existing examples with positive curvature. More precisely

\begin{maintheorem}\label{T:Eschenburg}
Let $M^7$ be a compact manifold, homotopy equivalent to a $7$-dimensional Eschenburg space. Then there are no submersions from $M^7$ to a manifold of dimension $\leq 3$.
\end{maintheorem}
\begin{maintheorem}\label{T:Bazaikin} Let $B^{13}$ be a compact manifold, homotopy equivalent to a Bazaikin space. Then there are no submersions from $B^{13}$ to a manifold of dimension $\leq 7$.
\end{maintheorem}

It follows that, were to exist a counterexample to the Petersen-Wilhelm conjecture, it would need to come from a new positively curved manifold, not even homotopy equivalent to any of the known ones. This shows the complexity of the problem, since finding new examples has been, and continues to be, a highly challenging task.

The theorems above, together with the results in \cite{AK}, imply the following.

\begin{maincorollary}
Any nontrivial submersion from a compact manifold, homotopy equivalent to any of the known compact positively curved manifolds, satisfies the Petersen-Wilhelm's conjecture.
\end{maincorollary}

Observe that the bounds in Theorems \ref{T:Eschenburg} and \ref{T:Bazaikin} are optimal, in the sense that the Eschenburg and the Bazaikin spaces admit submersions to $\CP^2$ and $\CP^4$ respectively. 

The proofs of Theorems A and B are remarkably similar, and follow from more general results. Recall that the Eschenburg and the Bazaikin spaces split rationally as $\mathbb{S}^2\times \mathbb{S}^5$ and $\CP^2\times \mathbb{S}^9$ respectively (cf. \cite{BK}). For each of these spaces, we apply tools from rational homotopy theory to rule out all potential submersions except those over $\mathbb{S}^2$ and $\CP^2$ respectively. Ruling out these last possibilities essentially constitutes the mathematical core of this note.

The paper is organized as follows. In Section \ref{S:prelim} we recall the main results from rational homotopy theory that will be used later on. Then in Sections \ref{S:Eschenburg} and \ref{S:Bazaikin} we prove Theorems \ref{T:Eschenburg} and \ref{T:Bazaikin}, respectively.

\begin{ack}
This project grew out of a conversation with Karsten Grove, while the first author was visiting the University of Notre Dame. The first author wishes to thank the University of Notre Dame for the hospitality during his stay.
\end{ack}

\section{Basics on rational homotopy theory}\label{S:prelim}

In this section we recall the basic facts about rational homotopy theory, that will be used later on. For more details on the theory, the reader is referred to \cite{FHT} and \cite{AP}. All algebras and vector spaces in this section are over the field $\QQ$ of rational numbers.

\subsection{Differential graded algebras}
Recall that a \emph{differential graded algebra} (also called $dga$) consists of a graded commutative algebra $A=\oplus_i A_i$, together with a degree-1 map (a \emph{differential}) $d:A\to A$ that satisfies the Leibniz rule, and $d\circ d=0$. A morphism of dga's $\varphi:(A,d)\to (B,d)$ is a graded algebra morphism $\varphi:A\to B$, commuting with the differentials.

Since $d\circ d=0$, it is possible to define the cohomology of a dga $H^*(A,d)$ As the homology of the complex $(A,d)$, and one defines the \emph{cohomological dimension} of $(A,d)$ to be the largest integer $n$ (possibly $\infty$) such that $H^n(A,d)\neq 0$.

\subsection{Sullivan functor}
The \emph{Sullivan functor} $A_{PL}$ is a contravariant functor that associates, to any path-connected pointed space $X$ (resp. to any map $f:X\to Y$ between path connected pointed spaces) a dga $(A_{PL}(X), d)$ given by the rational PL-forms on the singular simplices of $X$ (resp. a morphism $f^*:(A_{PL}(Y),d)\to (A_{PL}(X),d)$ of dga's obtained by pull-back of PL-forms). The Sullivan functor has the property that, for any path connected pointed space $X$, one has
\[
H^*(A_{PL}(X),d)\simeq H^*(X,\QQ).
\]

\subsection{Sullivan algebras} Given a $\ZZ$-graded vector space $V$, let $V^{even}$ (resp. $V^{odd}$) denote the subspace of $V$ spanned by the elements of even (resp. odd) degree, and define the free commutative graded algebra $\Lambda V$ by
\[
\Lambda V= \wedge V^{odd}\otimes \QQ[V^{even}],
\]
where $\wedge V^{odd}$ denotes exterior algebra and $\QQ[V^{even}]$ denotes polynomial algebra. A \emph{Sullivan algebra} is a dga of the form $(\Lambda V, d)$, satisfying
\[
d(\Lambda V)\subseteq \Lambda^+V\cdot \Lambda^+V
\]
where $\Lambda^+V\subset \Lambda V$ denotes the ideal generated by the elements of strictly positive degree.
\\

\subsection{Nilpotent spaces, and minimal model}\label{S:min-model}
Let now $X$ be a CW-complex. Recall that $X$ is called \emph{nilpotent} if $\pi_1(X)$ is nilpotent, and the action of $\pi_1(X)$ on $\pi_i(X)$ is nilpotent for all $i\geq 1$. For a nilpotent CW-complex $X$, there exists a space $X_\QQ$, called the \emph{rationalization} of $X$, such that $H^*(X_\QQ,\ZZ)\simeq H^*(X;\QQ)$, together with a map $X\to X_\QQ$ inducing isomorphisms $\pi_k(X)\otimes \QQ\to \pi_k(X_\QQ)\otimes \QQ$. Two nilpotent CW-complexes are \emph{rationally homotopy equivalent}, if the corresponding rationalizations are homotopy equivalent. Rationally homotopy equivalent spaces have, in particular, isomorphic rational homotopy groups $\pi_i^\QQ(X):=\pi_i(X)\otimes \QQ$ and isomorphic rational cohomology rings.

Given a nilpotent CW-complex, one can also define a Sullivan algebra $(\Lambda V_X,d)$ and a morphism $\varphi:(\Lambda V_X,d)\to (A_{PL}(X),d)$ of dga's, such that the induced map in cohomology $\varphi_*:H^*(\Lambda V_X,d)\to H^*(A_{PL}(X),d)\simeq H^*(X,\QQ)$ is an isomorphism. The algebra $(\Lambda V_X,d)$ is called \emph{Sullivan model} (also \emph{minimal model}) of $X$, and it satisfies the following remarkable properties:
\begin{itemize}
\item The minimal model is unique up to isomorphism.
\item If $\pi_1(X)$ is abelian, there are isomorphisms $V_X^i\simeq \pi_i^\QQ(X)$ for all $i\geq 1$, where $V_X^i\subset V_X$ is the subspace of degree $i$, and $\pi_i^\QQ(X)$ denotes $\pi_i(X)\otimes \QQ$ (cf. \cite[Theorem (2.3.7)]{AP}).
\item The minimal model determines $X$ up to rational homotopy equivalence: if $Y$ is another nilpotent CW-complex, its minimal model is isomorphic to $(\Lambda V_X,d)$ if and only if $X,Y$ are rationally homotopy equivalent.
\end{itemize}

\subsection{Minimal models and fibrations}\label{SS:relat.model}
Minimal models behave nicely with respect to fibrations. In particular, given a fibration $F\to M\to X$ of nilpotent spaces, where $F$ and $X$ have minimal models $(\Lambda V_F,d_F)$ and $(\Lambda V_X, d_X)$ respectively, then there exists a dga $(\Lambda V_F\otimes \Lambda V_X,D)$ and a morphisms of dga's $(\Lambda V_F \otimes \Lambda V_X,D)\to (A_{PL}(M),d)$ inducing isomorphism in cohomology (see \cite[Sec. 15(a)]{FHT}). This model comes together with dga maps

$$(\Lambda V_X,d_X)\stackrel{1\otimes id}{\longrightarrow} (\Lambda V_F\otimes \Lambda V_X,D)\stackrel{id\otimes \epsilon}{\longrightarrow} (\Lambda V_F,d_F),$$
where $\epsilon:\Lambda V_X\to \QQ$ sends $\Lambda^+V_X$ to $0$, which induce in cohomology the expected maps $H^*(X,\QQ)\to H^*(M,\QQ)\to H^*(F,\QQ)$.
 The dga $(\Lambda V_F \otimes \Lambda V_X,D)$ is not necessarily a minimal model, but it is called \emph{relative minimal model} of $M$.
\subsection{Rationally elliptic spaces}

Let $X$ be a nilpotent CW-complex, with $\pi_1(X)$ abelian. Then $X$ is called \emph{rationally elliptic}, if
\[
\sum_{i=0}^\infty\dim H^i(X,\QQ)<\infty,\qquad \sum_{i=1}^\infty\dim \pi_i^\QQ(X)\leq \infty.
\]

Friedlander and Halperin proved in \cite{FH} that rationally elliptic spaces have very restrictive rational homotopy groups. In particular, the list of possible rational groups of rationally elliptic spaces of dimension $\leq 7$, computed in \cite{Pav} (for dimension $\leq 6$) and \cite{DeV} (in dimension 7) is exceptionally limited, and it is shown in Table \ref{Table} at the end of this paper.

All the known compact simply connected manifolds with non-negative curvature are rationally elliptic (and in fact, the so-called Bott-Grove-Halperin conjecture states that this should always be the case). In particular, we will be dealing with submersions $\pi:M\to X$ where $M$ is a compact, simply connected, rationally elliptic manifold.

The following result is well known, and it will be crucial to rule out most submersions from Eschenburg and Bazaikin biquotients.

\begin{proposition}\label{P:rat-ell} Suppose $M$, $X$ are compact simply connected manifolds, and $\pi:M\to X$ is a submersion with fiber $F$. Then:

\begin{enumerate}[a.]
\item $F$ is a compact nilpotent space, with abelian fundamental group. %NOT NECESSARILY ABELIAN: there exist nontrivial nilpotent actions of Z on Z+Z, for example.
\item If $M$ is rationally elliptic, so are $X$ and $F$, and the dimension of $F$ can be computed from the ranks $f_i=\dim\pi_i^\QQ(F)$, by
\begin{align}\label{E:dimension-formula}
\dim F= \sum_{i=1}^\infty(2i+1)f_{2i+1}-\sum_{i=1}^\infty(2i-1)f_{2i}.
\end{align}
\end{enumerate}
\end{proposition}

\begin{proof}
a. Since $\pi:M\to X$ is a submersion between compact manifolds, it is also a locally trivial fiber bundle (in particular a fibration). The generic fiber, denoted by $F$, is a compact manifold and, being the fiber of a fibration from a nilpotent space, it is itself nilpotent (cf. \cite[Theorem 2.2]{HMR}). Moreover, from the long exact sequence in homotopy for the fibration $F\to M\to X$ it follows that $\pi_1(F)$ is abelian.
\\

b. It follows from the previous point that the spaces $M$, $X$, $F$ admit minimal models $(\Lambda V_M,d)$, $(\Lambda V_X,d)$, $(\Lambda V_F,d)$ respectively. Since $M$, $X$ are simply connected and $\pi_1(F)$ is abelian, it follows from Section \ref{S:min-model} that $V_M^i$, (resp. $V_X^i$, $V_F^i$) is isomorphic to $\pi_i^\QQ(M)$ (resp. $\pi_i^\QQ(X)$, $\pi_i^\QQ(F)$) for all $i\geq 1$.

Recall that, by the rational dichotomy (cf. for example Theorem A in \cite{FHTII}), a manifold $M$ is rationally elliptic if and only if the rational Betti numbers $b_i(\Omega M)$ of the loop space of $M$ grow at most polynomially in $i$. The $E_2$-page of the Leray-Serre spectral sequence for the fibration $\Omega M\to \Omega X\to F$ is given by $E_2^{p,q}=H^p(\Omega M,\QQ)\otimes H^q(F,\QQ)$, and therefore
\[
\dim H^i(\Omega X,\QQ)\leq \sum_{p-q=i} \dim H^p(\Omega M,\QQ) \dim H^q(F,\QQ)\leq C\sum_{k=0}^{\dim F}b_{i-k}(\Omega M)
\]
where $C=\max_{j} b_j(F)$. Because $M$ is rationally elliptic, by Proposition 33.9 in \cite{FHT} it follows that
\[
\sum_{k=0}^{\dim F}b_{i-k}(\Omega M)\leq B i^m
\]
for some constant $B$ and some positive integer $m$. Therefore, $b_i(\Omega X)\leq BC\, i^m$, and $\sum_{i=1}^k b_i(\Omega X)\leq BC\, k^{m+1}$, so $X$ is rationally elliptic. In particular, $\sum_i\dim(\pi_i^\QQ(X))$ is finite. From the long exact sequence in rational homotopy, the same holds for $F$ and thus $F$ is rationally elliptic.

Finally, since $F$ is rationally elliptic, Equation (2.2) in \cite{FH} gives that the largest integer $n_F$ such that $H^{n_F}(\Lambda V_F,d)=H^{n_F}(F,\QQ)\neq0$ can be computed from the $f_i=\dim V_F^i=\dim \pi_i^\QQ(F)$, as
\[
n_F= \sum_{i=1}^\infty(2i+1)f_{2i+1}-\sum_{i=1}^\infty(2i-1)f_{2i}.
\]
Since $F$ is compact and orientable, $n_F=\dim F$ and Equation \eqref{E:dimension-formula} follows.

\end{proof}

\section{Submersions from 7-dimensional examples}\label{S:Eschenburg}

The goal of this section is to prove Theorem \ref{T:Eschenburg}. This result will follow from the more topological Propositions \ref{P:Submersion-dim-7} and \ref{P:No-S2} below.

\begin{proposition}\label{P:Submersion-dim-7}
Let $M^7$ be a compact, simply connected manifold that is rationally homotopy equivalent to $\mathbb{S}^2\times \mathbb{S}^5$. If $\pi:M\to X$ is a submersion onto a simply connected compact manifold, then $X$ has the rational homotopy groups of either $\mathbb{S}^2, \CP^2, \mathbb{S}^5$, or $\mathbb{S}^2\times\CP^2$.
\end{proposition}

\begin{proof}

Since $M$ is rationally homotopy equivalent to $\mathbb{S}^2\times \mathbb{S}^5$, in particular it is rationally elliptic, and by Proposition \ref{P:rat-ell}, the same holds for $X$ and $F$.
%\[
%\pi_i^\QQ(M) = \begin{cases}
%    \QQ, & i=2,3,5,\\
%    0, & \textrm{otherwise}.
%  \end{cases}\qquad\qquad
%H^i(M,\QQ) = \begin{cases}
%    \QQ, & i=0,2,5,7,\\
%    0, & \textrm{otherwise}.
%  \end{cases}  
%\]
%In particular, $M$ is rationally elliptic and, by Proposition \ref{P:rat-ell}, the same holds for $X$ and $F$.
Because $X$ is simply connected, rationally elliptic and has $\dim X\leq 6$, there is a finite list of possibilities for the ranks $c_i=\dim\pi_i^\QQ(X)$ of the rational homotopy groups of $X$. We list these possibilities in Table \ref{Table}, together with the corresponding dimension $n$ of $X$, computed using Equation \eqref{E:dimension-formula}.

Suppose that the rational homotopy groups of $X$ are not the ones of $\mathbb{S}^2,\CP^2,\mathbb{S}^5$ or $\mathbb{S}^2\times\CP^2$. We will show by contradiction that such a space cannot occur.

For each of these possible $X$, one can easily compute the corresponding rational homotopy groups $\pi_i^\QQ(F)=\QQ^{f_i}$ of $F$ via the long exact sequence in (rational) homotopy for the fibration $F\to M\to X$. For instance, suppose that $X$ has the rational homotopy groups of $\mathbb{S}^3$. In this case the fiber $F$ is 4-dimensional, and the non-zero ranks $f_i=\dim \pi_i^\QQ(F)$ are either:
\begin{enumerate}
\item $f_2=f_5=1$, or
\item $f_3=f_5=1$, $f_2=2$.
\end{enumerate}

In the first case, the models for $X$ and $F$ would be of the type $(\Lambda(x_3),d)$ and $(\Lambda(y_2, y_5), d)$ respectively, and thus a relative model for $M$ (cf. Section \ref{SS:relat.model}) would be $(\Lambda(y_2,x_3,y_5), D)$ with $D(x_3)=0$. In this case however, $y_2^2$ would represent a nonzero element of $H^4(M,\QQ)$, in contradiction with the fact that $H^4(M,\QQ)=H^4(\mathbb{S}^2\times \mathbb{S}^5,\QQ)=0$.

In the second case, using Equation \eqref{E:dimension-formula}, one obtains that the cohomological dimension of $H^*(\Lambda V_F,d)\simeq H^*(F,\QQ)$ would be 6, in contradiction with the fact that the cohomological dimension of compact orientable manifolds agrees with the usual dimension.

In all remaining cases for the rational homotopy groups of $X$, one obtains a contradiction as in the second case in the example above. Namely, one computes all possible $f_i=\dim \pi_i^\QQ(F)$, and uses Equation \eqref{E:dimension-formula} to compute the cohomological dimension of $H^*(\Lambda V_F,d)\simeq H^*(F,\QQ)$, which never agrees with the actual dimension of $F$.

\end{proof}

Notice that all the possibilities listed above can occur, as quotients of $\mathbb{S}^2\times \mathbb{S}^5$.
%\begin{remark}
%When $X$ is compact, simply connected and has the rational homotopy groups of $\mathbb{S}^2$, (resp. $\CP^2$, $\mathbb{S}^5$, $\mathbb{S}^2\times\CP^2$) then using standard result in low dimensional topology and the results in \cite{Her}, we obtain that either:
%\begin{itemize}
%\item $X$ is diffeomorphic to $\mathbb{S}^2$,
%\item $X$ is homeomorphic to $\CP^2$,
%\item $X$ is rationally homotopy equivalent to $\mathbb{S}^5$, or
%\item $X$ has the real homotopy type of $\mathbb{S}^2\times \CP^2$.
%\end{itemize}
%
%\end{remark}

\begin{proposition} \label{P:No-S2}
Let $M^7$ be as in Proposition \ref{P:Submersion-dim-7} and assume in addition that $\pi_4(M)\neq \ZZ_2$.
%as in Proposition \ref{P:Submersion-dim-7}, and assume in addition that $M$ admits a principal $S^1$-bundle $P\to M$, with $P$ homotopy equivalent to $\SU(3)$.
Then $M$ does not admit submersions onto any manifold of dimension $\leq 3$.
\end{proposition}
\begin{proof}
Since $M$ is simply connected, we can restrict our attention to submersions only simply connected manifolds. By Proposition \ref{P:Submersion-dim-7}, the only possible sumbersion from $M$ is onto $\sph^2$.

Suppose that there is a submersion $\pi: M\to \mathbb{S}^2$, with fiber $F^5$. Since $M$ is rationally homotopy equivalent to $\mathbb{S}^2\times \mathbb{S}^5$, it follows from the long exact sequence in rational homotopy that the nonzero ranks $f_i$ of the rational homotopy groups $\pi_i^\QQ(F)=\QQ^{f_i}$ have to be one of the following:
\begin{enumerate}
\item $f_5=1$, or
\item $f_1=f_2=f_5=1$, or
\item $f_2=f_3=f_5=1$.
\end{enumerate}
The latter cannot occur, since by \eqref{E:dimension-formula}, the cohomological dimension of $F$ equals $7\neq 5$. 

To rule out the remaining cases, we need to consider the Leray-Serre spectral sequence in cohomology of the fibration $M\to \mathbb{S}^2$, with coefficients in $R=\QQ$ or $\ZZ$. Since the base of the fibration is $\mathbb{S}^2$, the elements in the $E_2$-page are nonzero only in the $0$-th and $2$-nd columns, and we obtain a long exact sequence of the form
\begin{align}
\label{E:LES} 0 & \to H^1(M,R)\to H^1(F,R)\to H^0(F,R)\\
  & \stackrel{\beta}{\to} H^2(M,R)\to H^2(F,R)\to H^1(F,R)\to H^3(M,R)\to \dots \nonumber
\end{align}

\smallskip

For the case $f_1=f_2=f_5=1$, the fact that $\pi_1^\QQ(F)=\QQ$ implies that $H^1(F,\QQ)=\QQ$, and we obtain the following contradiction. On the one hand, since $H^2(M,\QQ)=H^0(F,\QQ)=\QQ$ and $H^1(M,\QQ)=H^3(M,\QQ)=0$, it follows from the long exact sequence \eqref{E:LES} with rational coefficients that $H^2(F,\QQ)=\QQ^2$. On the other hand, the minimal model $(\Lambda V_F,d)$ of $F$ equals $(\Lambda(x_1, x_2, x_5),d)$ for some differential $d$, and in particular $\dim H^2(F,\QQ)=\dim H^2(\Lambda V_F,d)\leq \dim \Lambda^2V_F=1$.

\medskip

The rest of the proof is dedicated to rule out the possibility $f_5=1$. In this case $F$ has the same model as $\mathbb{S}^5$, and the groups $H^i(F,\ZZ)$ are finite for $i=1,\ldots ,4$. It follows  from the long exact sequence \eqref{E:LES} with integer coefficients that the map $\beta: H^0(F,\ZZ)\to H^2(M,\ZZ)$ is nonzero. Observe that under the identification $H^0(F,\ZZ)\cong H^2(\mathbb{S}^2,\ZZ)$ from the $E^2$-page of the spectral sequence, the map $\beta$ is equivalent to
\[
\pi^*:H^2(\mathbb{S}^2,\ZZ)=\ZZ\to H^2(M,\ZZ)=\ZZ.
\]
Therefore, the generator $g$ of $H^2(\mathbb{S}^2,\ZZ)$ is sent to $k\bar{g}$, where $\bar{g}$ is a generator of $H^2(M,\ZZ)$ and $k$ is a positive integer. 

Recall that principal $S^1$-bundles over a manifold $X$ are in bijective correspondence with the elements of $H^2(X,\ZZ)$, via the first Chern class. The generator $g$ corresponds to the Hopf fibration $\mathbb{S}^3\to \mathbb{S}^2$. Moreover, letting $P\to M$ denote the principal $S^1$-bundle with Chern class $\bar{g}$, it is easily checked from the Gysin sequence of $S^1\to P\to M$ that $P$ is 2-connected. By taking a cyclic subgroup $\ZZ_k\subset S^1$, the quotient $P/\ZZ_k$ is the total space of an $S^1$ bundle $P/\ZZ_k\to M$, with first Chern class $k\bar{g}$. Therefore, the bundle $\pi^*\mathbb{S}^3\to M$ is isomorphic to $P/\ZZ_k\to M$, and in particular $\pi^*\mathbb{S}^3$ is homeomorphic to $P/\ZZ_k$.

The submersion $\hat{\pi}:\pi^*\mathbb{S}^3\to \mathbb{S}^3$, whose fibers are diffeomorphic to $F$, lifts to a submersion $\tilde{\pi}:P'\to \mathbb{S}^3$ where $P'$ denotes the universal cover of $\pi^*\mathbb{S}^3$. Here $P'$ is homeomorphic to $P$, is 2-connected, and the fibers are diffeomorphic to the universal cover $\tilde{F}$ of $F$. Observe that $\tilde{F}$ is a simply connected rational sphere.

%
% is diffeomorphic to $F$, the fiber of $\pi$. Letting $P$ be the universal cover of $\pi^*\mathbb{S}^3$, the map $\hat{\pi}$ lifts to a submersion $\tilde{\pi}: P\to \mathbb{S}^3$ whose fiber is the universal cover $\tilde{F}$ of $F$. Observe that $\tilde{F}$ is a simply connected rational sphere.

Let $j:\tilde{F}\to P'$ be the inclusion of a fiber of $\tilde\pi$. Since $P'\to \mathbb{S}^3$ is a submersion, one has the well-known equality $j^*(w(P')) = w(\tilde{F})$ on the total Stiefel-Whitney classes. %\cf \cite{BorelHirzebruch:III}*{p.\,5.1}
However, since $P'$ is 2-connected, one has $w_2(P')=0$ and therefore $w_2(\tilde{F})=0$. In other words,  $\tilde{F}$ is a simply connected, compact, spin $5$-manifold.

It follows from Smale's classification of simply connected, spin $5$-manifolds that the torsion of $H_2(\tilde{F},\ZZ)=H^3(\tilde{F},\ZZ)$ is isomorphic to a sum $\bigoplus_i \ZZ_{k_i}\oplus\ZZ_{k_i}$. The long exact sequence \eqref{E:LES} in integer cohomology for $\tilde{\pi}:P'\to \mathbb{S}^3$ yields 
\[
\dots\to H^5(\tilde{F},\ZZ)\to H^{3}(\tilde{F},\ZZ)\to H^{6}(P',\ZZ)\to\dots
\]
Since $H^6(P',\ZZ)=H_2(P',\ZZ)=0$ and $H^5(\tilde{F},\ZZ)=\ZZ$, it follows that $H^3(\tilde{F},\ZZ)$ is finite and cyclic, thus in order not to contradict Smale's result, one must have $H^3(\tilde{F},\ZZ)=0$, and this implies that $\tilde{F}$ is diffeomorphic to the sphere $\mathbb{S}^5$. However, from the long exact sequence for $\tilde{F}\to P'\to \mathbb{S}^3$ we obtain $\pi_4(P')\simeq\pi_4(\mathbb{S}^3)=\ZZ_2$. From the long exact sequence of $S^1\to P'\to M$ it follows that $\pi_4(P')=\pi_4(M)$ and thus $\pi_4(M)\simeq \ZZ_2$, in contradiction with the assumption.

\end{proof}

\subsection*{Proof of Theorem \ref{T:Eschenburg}} Let $M$ be homotopy equivalent to an Eschenburg space $M'$. By Belegradek and Kapovitch (cf. \cite[Lemma 8.2]{BK}), $M'$ (and thus $M$) is rationally homotopy equivalent to $\sph^2\times\sph^5$, and by Proposition 31 in \cite{Esc}, $\pi_4(M)=\pi_4(M')=0$. The result now follows from Proposition \ref{P:No-S2}.

\section{Submersions from 13-dimensional examples}\label{S:Bazaikin}

The goal of this section is to prove Theorem \ref{T:Bazaikin}. This result will follow from the more topological Propositions \ref{P:Submersion-dim-13} and \ref{P:No-CP2} below.

\begin{proposition}\label{P:Submersion-dim-13}
Let $B^{13}$ be a compact, simply connected manifold that is rationally homotopy equivalent to $\mathbb{S}^9\times \CP^2$. If $\pi:B\to X$ is a submersion onto a simply connected compact manifold $X$ of dimension $\leq 7$, then $X$ has the rational homotopy groups of $\CP^2$.
\end{proposition}

\begin{proof}
This proposition is proved along the same lines as the proof of Proposition \ref{P:Submersion-dim-7}: since $B$ is rationally homotopy equivalent to $\mathbb{S}^9\times \CP^2$, in particular it is rationally elliptic and by Proposition \ref{P:rat-ell} so are $X$ and $F$. Since $X$ is compact, simply connected and with $\dim X\leq 7$, the rational homotopy groups of $X$ fall into the finite list in Table \ref{Table}. Assume moreover that $X$ does not have the rational homotopy groups of $\CP^2$. For any such case, we use the long exact sequence in rational homotopy for the submersion $B\to X$ to compute the ranks $f_i=\dim \pi_i^\QQ(F)$ of the fiber $F$. Finally, we use Equation \eqref{E:dimension-formula} to compute the cohomological dimension of $H^*(F,\QQ)$, which never agrees with the dimension of $F$ with the exception of one case, namely if $X$ has the rational homotopy groups of $\mathbb{S}^5$, and $f_2=f_9=1$.

In this case, the minimal models of $X$ and $F$ are respectively
\[
(\Lambda V_X,d)=(\Lambda (x_5),d=0),\quad (\Lambda V_F,d)= (\Lambda(z_2,z_9),dz_2=0, dz_9=z_2^5).
\]
Then a relative model for $B$ is $(\Lambda V_B,D)=(\Lambda(z_2,x_5,z_9),D)$, where $D(x_5)=0$. In this case, since $D(\Lambda^4V_B)=0$, $x_5$ represents a nonzero element in $H^5(\Lambda V_B,D)=H^5(B,\QQ)$, contradicting the fact that $H^5(B,\QQ)=H^5(\mathbb{S}^9\times \CP^2,\QQ)=0$.

\end{proof}

Observe that for $B^{13}$ as in Proposition \ref{P:Submersion-dim-13}, we have that $H^1(B,\ZZ)=0$ and $H^2(B,\ZZ)=\ZZ$.

\begin{proposition}\label{P:No-CP2}
Let $B^{13}$ be as in Proposition \ref{P:Submersion-dim-13}, and assume in addition that the truncated cohomology ring $H^{\leq 4}(B,\ZZ)$ is isomorphic to the integral cohomology of $\CP^2$. Then  if $B$ admits a submersion onto a manifold $X$ of dimension $\leq 7$, it must be $X\simeq \CP^2$ and $\pi_i(B)\simeq \pi_i(\sph^5)$ for $i=3,\ldots 8$.

%$B$ admits a principal $S^1$-bundle $E\to B$, where $E$ is homotopy equivalent to the total space of a nontrivial $\mathbb{S}^5$-bundle over $\mathbb{S}^9$. Then $B$ does not admit any submersion onto a simply connected manifold $X$ with the rational homotopy groups of $\CP^2$.
%as in Proposition \ref{P:Submersion-dim-7}, and assume in addition that $M$ admits a principal $S^1$-bundle $P\to M$, with $P$ homotopy equivalent to $\SU(3)$.

\end{proposition}

\begin{proof} Suppose there is a submersion $B\to X$ where $X$ has dimension at most $7$, where as usual we can assume that $X$ is simply connected. Then by Proposition \ref{P:Submersion-dim-13} the base space $X$ has the rational homotopy groups of $\CP^2$, and by Lemma 3.2 of  \cite{PP}, $X$ is homeomorphic to $\CP^2$.

From the long exact sequence of the fibration $B\to X$, the ranks $f_i$ of the rational homotopy groups of the fiber $F$ must be either:
\begin{enumerate}
\item $f_1=f_2=f_9=1$, or
\item $f_9=1$.
\end{enumerate}
In the first case, the minimal models of $X$ and $F$ are, respectively,
\[
(\Lambda (x_2,x_5),dx_2=0, dx_5=x_2^3),\quad (\Lambda(z_1,z_2,z_9),dz_1=dz_2=0, dz_9=z_2^5).
\]
A relative model for $B$ is then
\[
(\Lambda V_B,D)=(\Lambda(z_1,z_2,x_2,x_5,z_9),D),
\]
with $D(x_2)=0$, $D(x_5)=x_2^3$ and $D(z_1)=ax_2$ with $a\in\{0,1\}$ (cf. Section \ref{SS:relat.model}). Since $H^1(B,\QQ)=0$, it must be $D(z_1)=x_2$. Then, since $H^2(B,\QQ)=\QQ$, it must be $D(z_2)=0$. However, in this way $D(z_2^3)=0$, but $\Lambda^5V_B=\Span(x_5, z_2^2z_1,x_2^2z_1, z_2x_2z_1)$ and $D(\Lambda^5V_B)=\Span(x_2^3, x_2^2z_2, x_2z_2^2)$. In particular, $z_2^3$ does not lie in $D(\Lambda^5V_B)$, and therefore it defines a nonzero class in $H^6(\Lambda V_B,D)\simeq H^6(B,\QQ)$, in contradiction with the fact that $H^6(B,\QQ)= 0$.

For the rest of the proof we assume to be in the second case, and the argument goes along the same lines as the proof of Proposition \ref{P:No-S2}. Notice that $F$ is, in this case, a rational sphere. Using the Leray-Serre spectral sequence of the fibration $\pi: B\to X$, we deduce that the map $\pi^*:H^2(X,\ZZ)\to H^2(B,\ZZ)$ is nonzero. A generator of $H^2(X,\ZZ)\simeq \ZZ$ corresponds to the first Chern class of the Hopf fibration $\mathbb{S}^5\to \CP^2\simeq X$. Now, let $E\to B$ denote the circle bundle whose first Chern class is a generator of $H^2(B,\ZZ)\simeq \ZZ$, and observe that in particular $\pi_1(E)=0$. Then $\pi$ can be lifted, up to homotopy, to a fibration $p:E\to \mathbb{S}^5$, and whose fiber $\tilde{F}$ is the universal cover of $F$ and hence a simply connected rational sphere.

We can obtain some information about the topology of $E$ from the properties of the Gysin sequence of $S^1\to E\to B$. Since the first Chern class of a circle bundle equals its Euler class, it follows from the assumptions on the cohomology of $B$ that $H^i(E,\ZZ)=0$ for $i\leq 4$ and $H^5(E,\ZZ)=\ZZ$. In particular, $E$ is $4$-connected.

% The Leray-Serre spectral sequence in integral cohomology and the long exact sequence in rational homotopy for the fibration $\mathbb{S}^5\to E\to \mathbb{S}^9$ clearly yield:
%\[
%H^i(E,\ZZ) = \begin{cases}
%    \ZZ, & i=0,5,9,14\\
%    0, & \textrm{otherwise.}
%  \end{cases} 
%  \qquad\qquad
%  \pi_i^\QQ(E) = \begin{cases}
%    \QQ, & i=5,9\\
%    0, & \textrm{otherwise.}
%  \end{cases} 
%\]
%Moreover, if $\varphi:\mathbb{S}^5\to E$ denotes the inclusion of a fiber, then the map $\varphi_* :\pi_i(\mathbb{S}^5)\to\pi_i(E)$ is an isomorphism for $i\leq 7$.

Next we study the topology of the fiber $\tilde{F}$. Consider the Leray-Serre spectral sequence in integral cohomology of the fibration $\tilde{F}\to E\stackrel{p}{\to}\mathbb{S}^5$. The elements in the $E_2$-page are nonzero only in the $0$-th and $5$-th columns, thus we obtain a long exact sequence
\begin{align}
\label{E:LES over S5} 0 & \to H^4(E,\ZZ)\to H^4(\tilde{F},\ZZ)\to H^0(\tilde{F},\ZZ)\\
  & \to H^5(E,\ZZ)\to H^5(\tilde{F},\ZZ)\to H^1(\tilde{F},\ZZ)\to H^6(E,\ZZ)\to \dots \nonumber
\end{align}
It is straightforward to compute the cohomology of $\tilde{F}$ from \eqref{E:LES over S5}:
\[
H^i(\tilde{F},\ZZ) = \begin{cases}
    \ZZ, & i=0,9\\
    \ZZ_r, & i=5; \textrm{ for some integer } r\geq 1,\\
    0, & \textrm{otherwise.}
  \end{cases}  
\]
By Poincar\'e Duality, $\tilde{F}$ is 3-connected and $H_4(\tilde{F},\ZZ)=H^5(\tilde{F},\ZZ)=\ZZ_r$.

We claim that $\ZZ_r=0$. In fact, notice first that $\tilde{F}$ is a $(s-1)$-connected $(2s+1)$-manifolds, with $s=4$. By the seminal work by Wall on highly connected manifolds \cite[Cor. 2]{W}, there exists a nondegenerate bilinear form
\[
b:H_4(\tilde{F},\ZZ)\times H_4(\tilde{F},\ZZ) \to \QQ/\ZZ.
\]
which is, in this case, strongly skew-symmetric, which means that $b(x,x)=0$ for every $x\in H_4(\tilde{F},\ZZ)$. Since in this case $H_4(\tilde{F},\ZZ)\simeq \ZZ_r$ is cyclic, it must be $b=0$: in fact, for any $[m], [n]\in \ZZ_r$, we have $b([m],[n])=mn\cdot b([1],[1])=0$. However, since $b$ is nondegenerate, it follows that $\ZZ_r=0$, as claimed.

By Poincar\'e Duality, Hurewicz Theorem, and Whitehead Theorem, it then follows that $\tilde{F}$ is homotopy equivalent (hence homeomorphic) to a sphere $\mathbb{S}^9$ and, from the long exact sequence in homotopy for $\tilde{F}\to E\to \mathbb{S}^5$, we obtain that $\pi_i(E)\to \pi_i(\mathbb{S}^5)$ is an isomorphism for $i=1\ldots 8$. On the other hand, it follows from the long sequence in homotopy for $S^1\to E\to B$ that $\pi_i(B)\simeq \pi_i(E)$ for $i=3,\ldots 8$, and the result follows.

\end{proof}

\subsection*{Proof of Theorem \ref{T:Bazaikin}} Let $B$ be homotopy equivalent to a Bazaikin space $B'$. By Belegradek and Kapovitch (cf. \cite[Lemma 8.2]{BK}), $B'$ (and thus $B$) is rationally homotopy equivalent to $\CP^2\times\sph^9$. Moreover, one can check in \cite{FZ2} that the Bazaikin spaces satisfy the conditions on the cohomology ring required in Proposition \ref{P:No-CP2}, but $\pi_8(B')=0\neq \ZZ_{24}=\pi_8(\sph^5)$. The result now follows from Proposition \ref{P:No-CP2}.

\begin{table}[!htb]
  \begin{center}
    \renewcommand{\arraystretch}{1.1}
    $
    \begin{array}{|c||c|c|}\hline
      n & \text{ Example for } X^n  & c_i =\dim \pi_i^\QQ(X)\\ \hline\hline
      2 & \mathbb{S}^2 & c_2=c_3=1 \\\hline
      3 & \mathbb{S}^3 & c_3=1 \\\hline
      4 & \mathbb{S}^4 & c_4=c_7=1 \\
        & \CP^2 & c_2=c_5=1 \\
       & \mathbb{S}^2\times \mathbb{S}^2 \text{ or } \CP^2\sharp\CP^2 & c_2=c_3=2 \\\hline
      5 & \mathbb{S}^5 & c_5=1 \\
       & \mathbb{S}^2\times \mathbb{S}^3 & c_2=1, c_3=2 \\\hline
      6 & \CP^3 & c_2=c_7=1 \\
       & \mathbb{S}^3\times \mathbb{S}^3 & c_3=2 \\
       & \mathbb{S}^6 & c_6=c_{11}=1 \\
       & \mathbb{S}^2\times \mathbb{S}^4 & c_2=c_3=c_4=c_7=1 \\
       & \mathbb{S}^2\times \mathbb{S}^2\times \mathbb{S}^2 & c_2=c_3=3 \\
       & W^6 \text{ or } \mathbb{S}^2\times\CP^2 & c_3=c_5=1, c_2=2 \\\hline
      7 & \mathbb{S}^7 & c_7=1 \\
       & \mathbb{S}^3\times \mathbb{S}^4 & c_3=c_4=c_7=1 \\
       & \mathbb{S}^2\times \mathbb{S}^5 & c_2=c_3=c_5=1 \\
       & \mathbb{S}^2\times \mathbb{S}^2\times \mathbb{S}^3 & c_2=2, c_3=3 \\\hline
    \end{array}
    $
  \end{center}
  \caption{Possibilities for the rational homotopy groups of a low dimensional space.}\label{table: homotopy types}
  \label{Table}
\end{table}

\bibliographystyle{amsplain}

\begin{thebibliography}{10}

\bibitem{AP} C. Allday, V. Puppe, \emph{Cohomological methods in transformation groups}, Cambridge Studies in Advanced Mathematics, {\bf 32} (1993), Cambridge University Press, Cambridge.
\bibitem{AK} M. Amann, L. Kennard, \emph{Positive curvature and rational ellipticity}, Alg. \& Geom. Top. {\bf 15} (2015), no. 4, 2269--2301.
\bibitem{Baz}  Y. Bazaikin, \emph{On a certain family of closed 13-dimensional Riemannian manifolds of positive curvature}, Sib. Math. J. {\bf 37} (1996), no. 6, 1219--1237.
\bibitem{BK} I. Belegradek, V. Kapovitch, \emph{Obstructions to nonnegative curvature and rational homotopy theory}, J. Amer. Math. Soc. {\bf 16} (2003), no. 2, 259--284.
\bibitem{Che} X. Chen, \emph{Riemannian submersions from compact four manifolds}, Math. Z. {\bf 282} (2016), no. 1-2, 165--175. 
\bibitem{Dea} O. Dearricott, \emph{A 7-manifold with positive curvature}, Duke Math. J. {\bf 158} (2011), no. 2, 307--346.

\bibitem{DeV} J. De Vito, \emph{The classification of compact simply connected biquotients in dimension 6 and 7}, to appear in Math. Ann., DOI 10.1007/s00208-016-1460-8.
%\bibitem{Ehr} C. Ehresmann, \emph{Les connexions infinit\'esimales dans un espace fibr\'e diff\'erentiable}, S\'eminaire Bourbaki (1950), 29--55.
\bibitem{Esc} J.-H. Eschenburg, \emph{New examples of manifolds with strictly positive curvature}, Invent. Math. {\bf 66} (1982), 469--480.
\bibitem{FHTII} Y. Felix, S. Halperin, J.-C. Thomas, \emph{Elliptic Spaces II}, Enseign. Math. {\bf 39} (1993), no. 1-2, 25--32.
\bibitem{FHT} Y. Felix, S. Halperin, J.-C. Thomas, \emph{Rational homotopy theory}, Graduate Texts in Mathematics, Springer Ed.
\bibitem{FZ2} L. Florit, W. Ziller, \emph{On the topology of positively curved Bazaikin spaces} J. Eur. Math. Soc. (JEMS) {\bf 11} (2009), no. 1, 189--205.
\bibitem{FZ} L. Florit, W. Ziller, \emph{Orbifold fibrations of Eschenburg spaces}, Geom. Dedicata {\bf 127} (2007), no. 1, 159--175.
\bibitem{FH} J.B. Friedlander, S. Halperin, \emph{An arithmetic characterization of the rational homotopy groups of certain spaces}, Invent. Math. {\bf 53} (1979), 117--133. 
\bibitem{GAG} D. Gonz\'alez-\'Alvaro, L. Guijarro, \emph{Soft Restrictions on Positively Curved Riemannian Submersions}, J. Geom. Anal. {\bf 26} (2016), no. 2, 1442--1452.
\bibitem{GVZ} K. Grove, L. Verdiani, W. Ziller, \emph{An exotic $T^1\mathbb{S}^4$ with positive curvature}, Geom. Funct. Anal. 21 (2011), no. 3, 499–524.  
%\bibitem{H} A. Hatcher, \emph{Algebraic topology}, Cambridge University Press, Cambridge, 2002. %xii+544 pp.
%\bibitem{Her} M. Herrmann, \emph{Classification and Characterization of rationally elliptic manifolds in low dimensions}, preprint arXiv:1409.8036.
\bibitem{HMR} P. Hilton, G. Mislin, J. Roitberg, \emph{Localization of Nilpotent groups and Spaces}, North-Holland Mathematics Studies {\bf 15} (1975).
\bibitem{Ker} M. Kerin, \emph{Some new examples with almost positive curvature}, Geom. Top. {\bf{15}} (2011). 217--260.
\bibitem{PP} G. Paternain, J. Petean, \emph{Minimal entropy and collapsing with curvature bounded from below}, Invent. Math. {\bf 151} (2003), no.2, 415--450.
\bibitem{Pav} A.V. Pavlov, \emph{Estimates for the Betti numbers of rationally elliptic spaces}, Siberian Mathematical Journal, {\bf43} (2002), no. 6, 1080--1085.
\bibitem{Sp} L. Speran\c{c}a, \emph{On Riemannian foliations over positively curved manifolds.} Preprint arXiv:1602.01046.  
\bibitem{W} CTC. Wall, \emph{Classification problems in differential topology. VI. Classification of $(s-1)$-connected $(2s+1)$-manifolds.} Topology 6 1967 273--296. 
\bibitem{WZ} B. Wilking, W. Ziller, \emph{Revisiting homogeneous spaces with positive curvature}, to appear in J. Reine Ang. Math., DOI: https://doi.org/10.1515/crelle-2015-0053.
\bibitem{Zil} W. Ziller, \emph{Riemannian Manifolds with positive sectional curvature}, Geometry of manifolds with non-negative sectional curvature, 1–19, Lecture Notes in Math., 2110, Springer, Cham, 2014.

\end{thebibliography}

\hspace*{1em}\\
\end{document}